\documentclass[a4paper,12pt,reqno]{amsart}
\usepackage{amsmath}
\usepackage{amssymb}

\newtheorem{theorem}{Theorem}[section]
\newtheorem{lemma}[theorem]{Lemma}
\newtheorem{definition}[theorem]{Definition}
\newtheorem{proposition}[theorem]{Proposition}
\newtheorem{corollary}[theorem]{Corollary}
\newtheorem*{corollary*}{Corollary}
\theoremstyle{remark}
\newtheorem*{remark}{Remark}
\numberwithin{equation}{section}

\newcommand{\R}{\mathbb{R}}

\newcommand{\abs}[1]{\ensuremath{\lvert #1 \rvert}}

\begin{document}

\title{On the number of cycles in a random permutation}

\author[K. Maples]{Kenneth Maples}
\address{Institut f\"ur Mathematik\\ Universit\"at Z\"urich\\ Winterthurerstrasse 190\\ 8057-Z\"urich,
Switzerland}
\email{kenneth.maples@math.uzh.ch}

\author[A. Nikeghbali]{Ashkan Nikeghbali}
\address{Institut f\"ur Mathematik\\ Universit\"at Z\"urich\\ Winterthurerstrasse 190\\ 8057-Z\"urich,
Switzerland}
\email{ashkan.nikeghbali@math.uzh.ch}

\author[D. Zeindler]{Dirk Zeindler}
\address{
Sonderforschungsbereich 701\\
Fakult\"at f\"ur Mathematik\\
Universit\"at Bielefeld\\
Postfach 10 01 31\\
33501 Bielefeld\\
Germany} 
\email{zeindler@math.uni-bielefeld.de}


\maketitle

\thispagestyle{empty}

\begin{abstract}
We show that the number of cycles in a random permutation chosen according to generalized Ewens measure is normally distributed and compute asymptotic estimates for the mean and variance.
\end{abstract}

\section{Introduction}

Let $S_n$ denote the permutation group on $n$ letters. For each permutation $\sigma \in S_n$, we write $c_j(\sigma)$ for the number of disjoint cycles of length $j$ in $\sigma$.  For any permutation, we let $K_{0n}(\sigma) := \sum_{j=1}^n c_j(\sigma)$ denote the number of cycles in $\sigma$. 

We are interested in the statistics of permutations produced in a random way. Random (uniform) permutations and their cycle structures have received much attention and have a long history (see e.g. the first chapter  of \cite{ABT02} for a detailed account with references). The literature on the topic has  grown quickly in recent years in relation to mathematical biology and theoretical physics, where models of non-uniform permutations are considered (see e.g. \cite{BeUe09, BeUe10, BeUeVe11, ErUe11}).  We will restrict our attention to random permutations with cycle weights  as considered in the recent work of Betz, Ueltschi and Velenik \cite{BeUeVe11} and Ercolani and Uelstschi \cite{ErUe11}. These are families of probability measures on $S_n$ that are constant on conjugacy classes with the distribution
\[
  \mathbb{P}(\sigma) := \frac1{h_n n!} \prod_{j=1}^n \theta_j^{c_j(\sigma)}
\]
where $\theta_j \geq 0$ is a given sequence of weights and $h_n$ is a normalization constant. If $\theta_j = 1$ for all $j$ then this is the uniform measure on $S_n$, while if $\theta_j = \theta_0$ is constant then this gives Ewens measure, which plays an important role in mathematical biology.

A situation of interest which appears in the study of the quantum gas in statistical mechanics is when the asymptotic behavior of $\theta_j$ is fixed for large $j$ (see \cite{BeUeVe11} and \cite{ErUe11}). Natural important historical questions arise on the behavior of $c_j(\sigma)$ or  $K_{0n}(\sigma)$. For instance, it is known that under the Ewens measure, or in special cases of weighted random permutations, the cycle counts  $(c_j(\sigma))_j$ converge to independent Poisson distributions (see \cite{ABT02} for the Ewens measure and \cite{ErUe11} and \cite{NikZei11} for weighted random permutations). The case of $K_{0n}(\sigma)$ is in fact more delicate and less results are available in the general case with cycle weights. It is well known that under the Ewens measure $K_{0n}(\sigma)$ satisfies a central limit theorem (see \cite{ABT02} for details and historical references). The methods used in this case are very probabilistic and rely on the Feller coupling. However, the Feller coupling does not exist in the model of random permutations with cycle weights. Ercolani and Ueltschi (\cite{ErUe11}) used generating series and refined saddle point analysis to obtain some asymptotic estimates for the mean of  $K_{0n}(\sigma)$ in some special cases but were not able to prove any central limit theorem. In \cite{NikZei11} the second and third authors used generating series and singularity analysis to prove a central limit theorem and some large deviations estimates in the cases where the generating series exhibit some logarithmic singularities, but the important cases corresponding to subexponential and algebraic growth of the generating series were still open (see the corollaries below for a more precise statement). In this paper we propose yet another, but more elementary, approach based on Cauchy's integral theorem for analytic functions to solve these problems.

 More precisely, with a sequence $\theta = \{\theta_j\}_{j=1}^\infty$ fixed, we write
\[
  g_\theta(t) = \sum_{k=1}^\infty \frac{\theta_k}{k} t^k
\]
for the indicated generating function. 
We will always assume that the series for $g_\theta$ converges in a neighborhood of the origin. 
We will also require that $g_\theta$ satisfies a technical condition which we call $\log$-admissibility, which will be defined in Section~\ref{sec:estmgf}.

Our main result is the following.
\begin{theorem}[Central Limit Theorem for $K_{0n}$] \label{thm:main}
	Suppose that $g_\theta(t)$ is defined for $t \in [0,1)$ with $g_\theta(t) \to \infty$ as $t \to 1$. 
Suppose further that $g$ is $\log$-admissible (see Definition~\ref{def:admissible}). Then there are sequences $\mu_n$ and $\sigma_n$ such that
\[
  \frac{K_{0n} - \mu_n}{\sigma_n}
\]
converges to a standard normal distribution.
\end{theorem}
The mean and standard deviation can also be explicitly computed. For now we will state explicit results for two cases of interest, and defer the general result to Section~\ref{sec:totalnumber}.

\begin{corollary} \label{cor:genalgebraic}
    Let $g(t) = \gamma (1 - t)^{-\beta}$ for some $\beta > 0$ and $\gamma > 0$. Then $K_{0n}$ converges asymptotically to a normal distribution with mean
    \[
      \mu_n = (\beta \gamma)^{-\beta/(\beta + 1)} n^{\beta/(\beta+1)} (1 + o(1))
    \]
and variance
\[
  \sigma_n^2 = ((\beta \gamma)^{-1} - (\beta + 1)^{-1}) (\beta \gamma)^{1/(\beta + 1)} n^{\beta/(\beta + 1)} (1 + o(1)).
\]
\end{corollary}
\begin{corollary} \label{cor:exponential}
    Let $g(t) = \exp (1 - t)^{-\beta}$ for some $\beta > 0$. Then $K_{0n}$ converges asymptotically to a normal distribution with mean
    \[
      \mu_n = \frac{n}{(\log n)^{1 + 1/\beta}} (1 + o(1))
    \]
    and variance
    \[
      \sigma_n^2 = (2 + \beta^{-1}) \frac{n}{(\log n)^{2 + 1/\beta}} 
    \]
\end{corollary}

Our approach relies in the following well-known calculation of the moment generating function for $K_{0n}$.
\begin{proposition} \label{prop:mgf}
	We have the power series identity
	\[
	  \exp( e^{-s} g_\theta(t)) = \sum_{n=0}^\infty h_n \mathbb{E} \exp(-s K_{0n}) t^n
	\]
	for all $s \in \mathbb{R}$, with the conventions that $h_0=1$ and $K_{00}=0$.
\end{proposition}

This proposition follows immediately from Polya's enumeration theorem with a small calculation. More details can be found for instance in \cite[Section 4]{NikZei11}

\begin{remark}
In \cite{NikZei11} the characteristic function $\mathbb{E} \exp(i t K_{0n})$ of $K_{0n}$ was considered. In our new approach it is crucial to rather consider the Laplace transform $\mathbb{E} \exp(-s K_{0n})$ where the variable $s$ is real in order to be able to evaluate the relevant contour integrals.
\end{remark}

The outline of this article is as follows. In Section~\ref{sec:estmgf} we define $\log$-admissible $g(t)$ and derive a formula for the coefficients of its generating function. In Section~\ref{sec:totalnumber} we use the formula to compute asymptotics for $\mathbb{E} \exp(-s K_{0n})$, of which Theorem~\ref{thm:main} is a direct consequence. We also prove Corollary~\ref{cor:genalgebraic} and Corollary~\ref{cor:exponential}. Finally, in Section~\ref{sec:largedeviation} we show how the proof of Theorem~\ref{thm:main} can be modified to give large deviation estimates for $K_{0n}$.

\subsection*{Notation}

We will also freely employ asymptotic notation as follows; let $f$, $g$, and $h$ be arbitrary functions of a parameter $n$. 
Then we write $f = O(g)$ to indicate the existence of a constant $C$ and threshold $n_0$ such that for all $n > n_0$, $\abs{f(n)} \leq C \abs{g(n)}$; 
the constant and threshold may be different in each use of the notation. We also write $f = h + O(g)$ to indicate $\abs{f - h} = O(g)$. 
We similarly write $f = o(g)$ to indicate that
\[
  \lim_{n \to \infty} \frac{f(n)}{g(n)} = 0.
\]
It is convenient to also employ Vinogradov notation: we write $f \lesssim g$ (and equivalently $g \gtrsim f$) for $f = O(g)$.

\section{Estimates for the moment generating function} \label{sec:estmgf}

We shall now introduce the class of functions $g_\theta(t)$ where we can compute the asymptotic behavior of $\mathbb{E}\exp(-s K_{0n})$.
\begin{definition} \label{def:admissible}
Let $g(t) = \sum_{n=0}^\infty g_n t^n$ be given with radius of convergence $\rho > 0$ and $g_n \geq 0$. We say that $g(t)$ is \emph{$\log$-admissible} 
if there exist functions $a,b,\delta:[0,\rho) \to \R^+$ and $R : [0,\rho) \times (-\pi/2, \pi/2) \to \R^+$ with the following properties.
\begin{description}
  \item[approximation] For all $\abs{\varphi} \leq \delta(r)$ we have the expansion
        \begin{align}
          g(re^{i\varphi})
          =
          g(r) + i\varphi a(r)-\frac{\varphi^2}{2} b(r)
          + R(r,\varphi)
          \label{eq_admissible_expansion}
        \end{align}
        where $R(r,\varphi) = o(\varphi^3 \delta(r)^{-3})$ and the implied constant is uniform.
  \item[divergence] $a(r) \to \infty$, $b(r) \to \infty$ and $\delta(r) \to 0$ as $r \to \rho$.
  \item[width of convergence] For all $\epsilon > 0$, we have $\epsilon\delta^2(r)b(r) - \log b(r) \to \infty$ as $r \to \rho$.
  \item[monotonicity] For all $\abs{\varphi} > \delta(r)$, we have
        \begin{align}
            \Re g(re^{i\varphi}) \leq \Re g(re^{\pm i\delta(r)}).
        \end{align}
\end{description}
\end{definition}
These properties can be interpreted as a logarithmic analogue of Hayman-admissibility \cite[Chapter~VIII.5]{FlSe09}. 
In fact, if $g(t)$ is $\log$-admissible then $\exp g(t)$ is Hayman-admissible. 
We have also introduced the $\epsilon$ term in the width condition, which is required for uniformity in the error term of Proposition~\ref{prop:generalasymptotic}.

The approximation condition allows us to compute the functions $a$ and $b$ exactly.  We have
\begin{align}
  a(r) &= rg'(r), \\
  b(r) &= rg'(r) +r^2 g''(r).
\end{align}
Clearly $a$ and $b$ are strictly increasing real analytic functions in $[0, \rho)$. 
The error in the approximation can similarly be bounded, so that
\[
  R(r,\varphi) \lesssim rg'(r) + 3r^2 g''(r) + r^3 g'''(r).
\]

Note that for Ewens measure ($\theta_j = 1$ for all $j \geq 0$), we have $g(r) = - \log(1 - r)$, so we can compute
\begin{align*}
 b(r) = \frac{r}{(1-r)^2} \text{ and } R(r,\varphi) \approx \frac{r^2 + r}{(1-r)^3}.
\end{align*}
Therefore $g(r) = - \log(1-r)$ is not $\log$-admissible and we cannot apply this method to such distributions.

We will frequently require the inverse function of $a$ on the interval $[0,\rho)$, and define $r_x$ to be the (unique) solution to $a(r) = x$ there.
It is easy to see that $r_x$ is strictly increasing real analytic function in $x$ and $r_x \to \rho $ if $x$ tends to infinity.
 
We now define for $s \in \R$ the sequence $G_{n,s}$ by
\begin{align}\label{eq:generating_general}
  \sum_{n=0}^\infty G_{n,s} t^n = \exp (e^{-s} g(t)).
\end{align}

If $g(t)$ is $\log$-admissible, then we can compute the asymptotic behavior of $G_{n,s}$ for $n\to \infty$.
The generating function in Proposition~\ref{prop:mgf} has the same form as \eqref{eq:generating_general} and we thus can 
compute also the asymptotic behavior of the moment generating function for $K_{0n}$ if $g_\theta$ is $\log$-admissible.

\begin{proposition} \label{prop:generalasymptotic}
Let $s\in \R$ and let $g$ be $\log$-admissible with associated functions $a$, $b$. 
Let further be $r_x$ to be the (unique) solution of $a(r) = x$.

Then $G_{n,s}$ has the asymptotic expansion
\[
 G_{n,s} = \frac{1}{\sqrt{2 \pi}} e^{s/2} r_{e^s n}^{-n} b(r_{e^s n})^{-1/2} \exp(e^{-s} g(r_{e^s n})) (1 + o(1))
\]
as $n \to \infty$ with the implied constant uniform in $s$ for $s$ bounded.
\end{proposition}

\begin{proof}
We apply Cauchy's integral formula to the circle $\gamma$ centered at $0$ of radius $r$. We get
\begin{align*}
  G_{n,s} &= \frac{1}{2 \pi i} \int_{\gamma} \exp( e^{-s} g(z) ) \, \frac{dz}{z^{n+1}}
  \\
  &= \frac{1}{2\pi r^n} \int_{-\pi}^\pi \exp( e^{-s} g(re^{i\varphi}) - in\varphi ) \, d\varphi
\end{align*}
with $r = r_{e^s n}$. We now split the integral into the parts $[-\delta(r),\delta(r)]$ and $[-\pi, -\delta(r)) \cup (\delta(r), \pi]$. 
We first look at the minor arc $[-\delta(r),\delta(r)]$. By hypothesis on $g$ we can expand in $\varphi$, giving
\begin{align*}
I_1 &:= \int_{-\delta(r)}^{\delta(r)} \exp( e^{-s} g(re^{i\varphi}) - in\varphi ) \, d\varphi \\
&= \int_{-\delta(r)}^{\delta(r)}\exp( e^{-s} (g(r) + i\varphi a(r) - b(r)\varphi^2/2 + o(\varphi^3 \delta(r)^{-3}) - in\varphi)) \, d\varphi.
\end{align*}

We have $e^{-s} a(r_{e^s n}) = n$ since $r = r_{e^s n}$, which cancels the linear terms in $\varphi$. We get
\[
  I_1 = \int_{-\delta(r)}^{\delta(r)}\exp( e^{-s} (g(r) - b(r)\varphi^2/2 + \varphi^3 o(\delta(r)^{-3}) )) \, d\varphi.
\]
We now observe that $\varphi^3 o(\delta(r)^{-3}) = o(1)$ on $[-\delta(r), \delta(r)]$ as $r = r_{e^s n} \to \rho$, or equivalently as $n \to \infty$. Rearranging, we get
\[
  I_1 = \exp(e^{-s} g(r)) \int_{- \delta(r)}^{\delta(r)} \exp(-e^{-s} b(r) \varphi^2 / 2) \, d\varphi (1 + o(1)).
\]
By assumption on the width of convergence of $g$, the integral converges to
\[
  e^{s/2} b(r)^{-1/2} \int_{- \delta(r) e^{-s/2} b(r)^{1/2}}^{\delta(r) e^{-s/2} b(r)^{1/2}} \exp(-x^2/2) \, dx = \sqrt{2 \pi} e^{s/2} b(r)^{-1/2} (1 + o(1))
\]
so that
\[
  I_1 = \sqrt{2 \pi} \exp(e^{-s} g(r)) e^{s/2} b(r)^{-1/2} (1 + o(1)).
\]

Next we estimate the integral over the major arc. By the monotonicity assumption on $g$, we have
\[
  I_2 := \int_{[-\pi, -\delta(r)] \cup [\delta(r), \pi]} \exp(e^{-s} g(r e^{i\varphi}) - i n \varphi) \, d\varphi \leq 2 \pi \exp( \Re (e^{-s} g(r e^{i \delta(r)}) ))
\]
We can apply the approximation for $g$ at $\varphi = \delta(r)$, giving
\[
  \Re g(r e^{i \delta(r)}) = g(r) - \frac{\delta(r)^2}{2} b(r) + o(1)
\]
Collecting terms and rearranging,
\begin{align}\label{eq:I_2_estimates}
  I_2 &\leq 2 \pi \exp(e^{-s} g(r)) b(r)^{-1/2} \exp(- e^{-s} \delta(r)^2 b(r) / 2 + \frac12 \log b(r)) \nonumber\\
  &= o( \exp(e^{-s} g(r)) b(r)^{-1/2} )
\end{align}
where at the last step we used the width of approximation for $g$. 

Combining the estimates for $I_1$ and $I_2$, we find that
\[
  G_{n,s} = \frac{1}{\sqrt{2 \pi}} r_{e^{s} n}^{-n} e^{s/2} \exp(e^{-s} g(r_{e^s n})) (1 + o(1)).
\]
where the error term is uniform in $s$ for $s$ in a fixed compact set.
Note that all error-terms in this proof are uniform in $s$ for $s$ in a fixed compact set.
\end{proof}
Note that the $\epsilon$ in the definition of $\log$-admissibility is required to make the error in \eqref{eq:I_2_estimates} uniform in $s$.

\section{The total number of cycles} \label{sec:totalnumber}

We are now ready to compute the asymptotic number of cycles in a random permutation as described in the introduction. 
We will restrict our attention to those examples in \cite{ErUe11} where the limiting behavior was not known, namely where the generating function $g_\theta$ is of the form
\[
  g_\theta(r) = \gamma (1 - r)^{-\beta}
\]
or
\[
  g_\theta(r) = \exp (1 - r)^{-\beta}.
\]
We will refer to such functions as exhibiting algebraic and sub-exponential growth, respectively.

We begin by observing that a formula for the moment generating function of $K_{0n}$ can be determined from Proposition~\ref{prop:generalasymptotic}.

\begin{corollary} \label{cor:mgfexpansion}
Let $s\in \R$, $g_\theta(t)$ be $\log$-admissible with associated functions $a$, $b$. 
Let further  $r_x$ be the (unique) solution to $a(r) = x$.

Then
\[
  h_n = \frac{1}{\sqrt{2\pi b(r_n)} r_n^n} \exp(g_\theta(r_n)) (1 + o(1))
\]
and
\[
\mathbb{E} \exp(-s K_{0n})  =  e^{s/2} \left( \frac{r_n}{r_{e^s n}} \right)^n 
\frac{\exp(e^{-s} g_\theta(r_{e^s n})) }{\exp(g_\theta(r_n))}
\left( \frac{b(r_n)}{b(r_{e^s n})} \right)^{1/2} (1 + o(1)),
\]
where the error terms are uniform in $s$ for $s$ in a fixed compact set.
\end{corollary}

\begin{proof}
By Proposition~\ref{prop:mgf}, we have
\[
  h_n \mathbb{E} \exp(-s K_{0n}) = G_{n,s}.
\]
Apply Proposition~\ref{prop:generalasymptotic} at $0$ to get the desired formula for $h_n$, and apply it again at $s$ to find the formula for $\mathbb{E} \exp(-s K_{0n})$.
\end{proof}

\subsection{A simple example}

Before we consider more complicated functions, we will illustrate the method with $g_\theta$ given by the equation
\[
  g_\theta(t) = \frac{1}{1 - t}.
\]
This generating function corresponds to the sequence $\theta_k = k$.

Our first step is to compute the moment generating function of $K_{0n}$ by finding an asymptotic expansion of the formula in Corollary~\ref{cor:mgfexpansion}.

\begin{proposition} \label{prop:algebraic}
Let $g_\theta(t) = (1-t)^{-1}$.  Then $g_\theta$ is admissible,
\[
  h_n = \frac{1}{\sqrt{4 \pi}} n^{-3/4} \exp(2 \sqrt{n} (1 + o(1))),
\]
and
\[
  \mathbb{E} \exp(-s K_{0n}) = e^{-s/4} \exp(2 \sqrt{n} (e^{-s/2} - 1) (1 + o(1)))
\]
where the errors are uniform in $s$ for $s$ in a fixed compact set.
\end{proposition}

We will prove Proposition~\ref{prop:algebraic} in a moment. We first see how to use this result to prove a central limit theorem for $K_{0n}$.

\begin{corollary} \label{cor:algebraic}
Let $g_\theta(t) = (1-t)^{-1}$. Then
\[
  \frac{K_{0n} - n^{1/2}}{2^{-1/2} n^{1/4}} \stackrel{d}{\longrightarrow} N.
\]
where $N$ is the standard normal distribution.
\end{corollary}

\begin{proof}[Proof of Corollary~\ref{cor:algebraic}]
It suffices to show that the moment generating function of the renormalized cycle count converges to $e^{s^2/2}$ for bounded $s$ (in fact, we only need this for $s$ sufficiently small, but our theorem gives a stronger result). We apply Theorem~\ref{prop:algebraic} at $s / (2^{-1/2} n^{1/4})$ to find
\[
  \mathbb{E} \exp(-s \frac{K_{0n}} {2^{-1/2} n^{1/4}}) = \exp(- \sqrt{2} s n^{1/4} + s^2/2 + O(s^3 n^{-1/4})).
\]
Here we used the uniformity of the error for bounded $s$. Multiplying both sides by $\exp(\sqrt{2} s n^{1/4} )$ and letting $n$ tend to $\infty$ completes the proof.
\end{proof}

\begin{proof}[Proof of Theorem~\ref{prop:algebraic}]
We first show that $g_\theta(t) = (1 - t)^{-1}$ is admissible. The monotonicity condition is obvious. We compute
\[
  a(t) = \frac{t}{(1 - t)^2}
\]
and
\[
  b(t) = \frac{t}{(1 - t)^2} + \frac{2 t^2}{(1 - t)^3}
\]
and note that $a, b \to \infty$ as $t \to 1$ from the left. It suffices to choose a function $\delta(t)$ that satisfies the remaining hypotheses. We observe that the width condition on $\delta$ is satisfied if $\delta(t) \lesssim (1 - t)^{3/2 - \eta}$ for some $\eta > 0$. Likewise, for the error in the approximation condition to be satisfied we need $\delta(t)^3 (1 - t)^{-4} = o(1)$. It therefore suffices to choose $\delta(t) = (1 - t)^\alpha$ for any $4/3 < \alpha \leq 3/2$.

We first calculate $r_x$. By definition we have
\[
  \frac{r_x}{(1 - r_x)^2} = x,
\]
which can be inverted to find
\[
  r_x = 1 - x^{-1/2} (1 + o(1)).
\]
We then compute
\[
  g(r_x) = \sqrt{x} (1 + o(1))
\]
and
\[
  b(r_x) = 2 x^{3/2} (1 + o(1)).
\]
With the approximation $(1 - \eta)^n \sim \exp(- \eta n)$ for $\eta$ sufficiently small, we apply Corollary~\ref{cor:mgfexpansion} to find
\[
  h_n = \frac{1}{\sqrt{4 \pi}} n^{-3/4} \exp(2 \sqrt{n} (1 + o(1)))
\]
and
\[
  \mathbb{E} \exp(-s K_{0n}) = e^{-s/4} \exp(2 \sqrt{n} (e^{-s/2} - 1) (1 + o(1)))
\]
as required.
\end{proof}

\subsection{The general case}

The previous calculation suggests how to transform the formula in Corollary~\ref{cor:mgfexpansion} into a form that is easier to manage. We will restrict our attention to those functions $g$ where the induced functions $r_x$ satisfy a family of inequalities
\begin{equation}
  \abs{r_x^{(k)} x^k} \lesssim \abs{r_x^{(k-1)} x^{(k-1)}} \label{eqn:technicalhelper}
\end{equation}
for all $x$ sufficiently large and $2 \leq k \leq 4$. This is easy to verify in practice and avoids some technical details for functions $g$ which diverge slowly at $1$ (i.e.~slower than $(1 - r)^{-\epsilon}$ for any $\epsilon > 0$).

It is convenient to define the functions
\[
  \eta_k(x) := (-1)^k \left. \frac{\partial^k}{\partial^k s} \right|_{s=0}  \log( r_{e^sx}).
\]
For example, this gives
\[
  \eta_1(x) =  -x \frac{r_x'}{r_x} 
\ \text{ and } 
\eta_2(x) =  - \eta_1(x) + x^2 \left( \frac{r_x''}{r_x} + \left( \frac{r'_x}{r_x} \right)^2  \right).
\]

\begin{proposition} \label{prop:expand}
Let $g$ be $\log$-admissible and $r_x$ be defined as above, satisfying condition \ref{eqn:technicalhelper}. Fix $M > 0$. Then for every $-M < s < M$, we have
\begin{align*}
\log \mathbb{E} \exp(-s K_{0n}) = &- g(r_n) (1 + o(1)) s \\
                                  &+ (g(r_n) + n\eta_1(n)) (1 + o(1)) s^2/2\\
                                  &+ O(\xi(n)s^3)
\end{align*}
where
\[
  \xi(n) = \sup_{x\in[e^{-M} n,e^M n]} g(r_x) + x\abs{\eta_1(x)}.
\]
\end{proposition}

\begin{proof}
We use Corollary~\ref{cor:mgfexpansion} and expand each factor with respect to $s$.

We start with the first term. Taking logarithms, we find
\begin{align*}
  \log (r_n r^{-1}_{e^s n})^n &= n \log(r_n) - n \log(r_{e^s n}) \\
  &= n\eta_1(n) s  - n\eta_2(n) s^2/2 + O(\xi_1(n) s^3) .
\end{align*}
where
\[
\xi_1(n) := \sup_{x\in[e^{-M} n,e^M n]} |\eta_3(x)|.
\]
Next we consider $\exp(e^{-s} g(r_{e^s n}) - g(r_n))$. We use $x = a(r_x) = r_x g'(r_x)$ and get
\[
  \frac{\partial}{\partial s} g(r_{e^s n}) = g'(r_{e^s n})r'_{e^s n} e^s n = e^{2s} n^2 \frac{r'_{e^s n}}{r_{e^s n}} = - e^s n \cdot \eta_1(e^s n).
\]
This gives 
\begin{align*}
e^{-s} g(r_{e^s n}) - g(r_n) = &- (n \eta_1(n) + g(r_n)) s\\
                               &+ (n \eta_1(n) + n \eta_2(n) + g(r_n)) s^2/2 \\
                               &+ O(\xi_2(n)s^3)
\end{align*}
where
\[
  \xi_2(n) = \sup_{x\in[e^{-M} n,e^M n]} g_r(x) + x\abs{\eta_1(x)} + x\abs{\eta_2 (x)} + x\abs{\eta_3(x)}.
\]

Finally we consider $(b(r_n) b(r_{e^s n})^{-1})^{1/2}$. Writing $b$ in terms of $a$, we get
\[
  b(r_x) = a(r_x) + r_x^2 g''(r_x) = x + r_x^2 g''(r_x)
\]
Differentiating the defining equation $  x = r_x g'(r_x) $ in $x$, we find that
\[
 1 = r_x' g'(r_x) + r_x r_x' g''(r_x) 
\]
so that
\[
r_x^2 g''(r_x) = \frac{r_x}{r_x'} - x
\]
and thus
\[
 b(r_x) = \frac{r_x}{r_x'} = \frac{-x}{\eta_1(x)}.
\]

This then gives
\[
 \log \left( \frac{b(r_n)}{b(r_{e^s n})} \right)^{1/2}
 =
s/2 - \frac{\eta_2(n)}{2\eta_1(n)} s + \left(  \frac{\eta_3(n)}{2\eta_1(n)} - \frac{\eta_2^2(n)}{2\eta_1^2(n)} \right) \frac{s^2}{2}
+ O( \xi_3(n) s^3)
\]
where
\[
  \xi_3(n) = \sup_{x\in[n,e^s n]} \left|\frac{\eta_4(x)}{ \eta_1(x)} \right|  + \left|\frac{\eta_2(x)\eta_3(x)}{ \eta_1^2(x)} \right| + \left|\frac{\eta_2^3(x)}{ \eta_1^3(x)} \right|. 
\]

We now apply the technical assumption to see that $\eta_k(n) \lesssim \eta_1(n)$ for $k = 2, 3, 4$. In particular, we see that the $(b(r_n) b(r_{e^s n})^{-1})^{1/2}$ is dominated by the other terms. Collapsing other redundant terms, we find that
\begin{align*}
\log \mathbb{E} \exp(-s K_{0n}) = &- g(r_n) (1 + o(1)) s \\
                                  &+ (g(r_n) + n\eta_1(n)) (1 + o(1)) s^2/2\\
                                  &+ O(s^3 \sup_{x\in[e^{-M} n,e^M n]} g(r_x) + x\abs{\eta_1(x)})
\end{align*}
as desired.
\end{proof}

As above, this has the following immediate corollary.

\begin{corollary} \label{cor:convergencetheorem}
Let $\theta$ be the defining sequence for a generalized Ewens measure. Suppose that $g_\theta$ is $\log$-admissible and $r_x$ satisfies the technical condition \ref{eqn:technicalhelper}. Then there are functions $\mu_n$ and $\sigma_n$ such that
    \[
      \frac{K_{0n} - \mu_n}{\sigma_n} \stackrel{d}{\longrightarrow} N(0,1)
    \]
where $\mu_n$, $\sigma_n$ satisfy the asymptotics
\[
  \mu_n = g(r_n) (1 + o(1))
\]
and
\[
  \sigma_n^2 = (g(r_n) + n\eta_1(n))(1 + o(1)).
\]
\end{corollary}

\begin{proof}
The only thing to check is whether the coefficient of $s^3$ is bounded by $(g(r_x) + x\abs{\eta_1(x)})^{3/2}$, but this is obvious.
\end{proof}

\subsection{Computing $g(r_n)$ and $\eta_1(n)$}

We now give two examples of how to apply Corollary~\ref{cor:convergencetheorem} to compute explicit asymptotics for given $g_\theta$. First we prove Corollary~\ref{cor:genalgebraic}.

\begin{proposition} \label{prop:generalalgebraic}
Let $g_\theta(t) = \gamma (1 - t)^{-\beta}$ for some $\beta > 0$ and $\gamma > 0$. Then $g_\theta$ is $\log$-admissible, $r_x$ satisfies the technical condition \ref{eqn:technicalhelper}, and there are asymptotic expansions
\[
  g_\theta(r_n) = (n \beta^{-1} \gamma^{-1})^{\frac{\beta}{\beta + 1}} (1 + o(1))
\]
and
\[
  n \eta_1(n) = - \frac{(\beta \gamma)^{1/(\beta + 1)}}{\beta + 1} n^{\beta/(\beta + 1)} (1 + o(1)).
\]
\end{proposition}
\begin{proof}
Admissibility follows once we construct an explicit $\delta(t)$. It suffices to find a function that satisfies the inequalities
\[
  \epsilon \delta(t)^2 b(t) - \log b(t) \to \infty
\]
for all $\epsilon > 0$, and
\[
  \delta(t)^3 R(t,\varphi) \to 0.
\]
For $t \to 1$, we have the lower bound $b(t) \gtrsim_{\gamma, \beta} (1 - t)^{-\beta-2}$ and the upper bound $R(t,\varphi) \lesssim_{\gamma, \beta} (1 - t)^{-\beta-3}$. Thus we see that any $\delta$ of the form $\delta(t) = (1 - t)^\alpha$ with $1 + \frac{\beta}{3} \leq \alpha < 1 + \frac{\beta}{2}$ suffices.

We compute $r_n$ by inverting $n = a(r_n) =  \beta \gamma r_n (1 - r_n)^{-\beta-1}$, so that
\[
  r_n = 1 - (\beta \gamma n^{-1})^{\frac{1}{\beta + 1}} (1 + o(1)).
\]
The derivatives of $r_n$ can be approximated in an analogous way, so that
\[
  \abs{r_n^{(k)}} = C_{\beta, \gamma} n^{\frac{-1}{\beta + 1} - k}
\]
and the technical condition is clear. We then estimate
\[
  g_\theta(r_n) = (n \beta^{-1} \gamma^{-1})^{\frac{\beta}{\beta + 1}} (1 + o(1)).
\]
and
\[
  n \eta_1(n) = - \frac{(\beta \gamma)^{1/(\beta + 1)}}{\beta + 1} n^{\beta/(\beta + 1)} (1 + o(1))
\]
with our estimate for $r_n$ and $r_n'$.
\end{proof}

Next we prove Corollary~\ref{cor:exponential}.

\begin{proposition} \label{prop:subexponential}
    Let $g_\theta(r) = \exp (1 - r)^{-\beta}$ for some $\beta > 0$. Then $g_\theta$ is $\log$-admissible, $r_x$ satisfies the technical condition \ref{eqn:technicalhelper}, and there are asymptotic expansions
    \[
      g_\theta(r_n) = \frac{n}{(\log n)^{1 + 1/\beta}} (1 + (\log n)^{-1/\beta} + (1 + \beta^{-1}) \frac{\log \log x}{\log x}(1 + o(1)))
    \]
    and
    \[
      n \eta_1(n) = - \frac{n}{(\log n)^{1 + 1/\beta}} (1 + (\log n)^{-1/\beta} - (1 + \beta^{-1}) \frac{1}{\log x}(1 + o(1)))
    \]
\end{proposition}

\begin{proof}
First we verify that $g_\theta$ is admissible. Monotonicity is obvious. We compute
\[
  a(r) = r g'(r) = r \beta (1 - r)^{-\beta - 1} \exp (1 - r)^{-\beta}
\]
and
\begin{align*}
  b(r) &= r g'(r) + r^2 g''(r) \\
       &= r^2 (\beta (\beta + 1) (1 - r)^{-\beta-2} + \beta^2 (1 - r)^{-2 \beta - 2}) \exp(1 - r)^{-\beta} \\
       &= r^2 \beta^2 (1 - r)^{-2 \beta - 2} \exp (1 - r)^{-\beta} (1 + O_\beta(1 - r)^\beta).
\end{align*}
These diverge at $1$. Once we estimate
\begin{align*}
  R(r,\varphi) &= r g'(r) + 3 r^2 g''(r) + r^3 g'''(r) \\
  &= r^3 \beta^3 (1 - r)^{-3 \beta - 3} \exp (1 - r)^{-\beta} (1 + O_\beta(1 - r)^\beta)
\end{align*}
we see that it remains to choose any $\delta$ that satisfies the pair of inequalities
\[
  \epsilon \delta(r)^2 r^2 \beta^2 (1 - r)^{-2 \beta - 2} \exp(1 - r)^{-\beta} \to \infty
\]
and
\[
  r^3 \beta^3 (1 - r)^{-3 \beta - 3} \exp (1 - r)^{-\beta} \lesssim \delta(r)^{-3}.
\]
Any $\delta$ of the form
\[
  \delta(r) = \exp (\alpha (1 - t)^{-\beta})
\]
with $1/3 \leq \alpha < 1/2$ suffices.

We next need an asymptotic approximation for $r_x$. This is provided by the following lemma.
\begin{lemma} \label{lem:expandradius}
    Let $f(x) := (1 - r_x)^{-\beta}$. Then we have the asymptotic expansion
\begin{multline*}
  f(x) = \log x - (1 + \beta^{-1}) \log \log x - \log \beta \\ + ((\log x)^{-1/\beta} + (1 + \beta^{-1}) \frac{\log \log x}{\log x})(1 + o(1))
\end{multline*}
Furthermore, we have the estimates
\[
  f^{(k)}(x) = (-1)^{k+1} (k-1)! \frac{1}{x^k} (1 - \frac{1 + \beta^{-1}}{\log x}) + O_{\beta,k} (\frac{\log \log x}{x^k (\log x)^2})
\]
\end{lemma}

\begin{proof}
    Once we make the substitution $f(x) = (1 - r_x)^{-\beta}$ in the equation $a(r_x) = x$, we see that $f$ is implicitly defined by the equation
    \[
      x = \beta (1 - f(x)^{-1/\beta}) f(x)^{1 + 1/\beta} \exp f(x).
    \]
    We then substitute $f(x) = \log x - (1 + \beta^{-1}) \log \log x + w$ and observe that $w = ((\log x)^{-1/\beta} + (1 + \beta^{-1}) \frac{\log \log x}{\log x})(1 + o(1))$.
    
    For the estimates on the derivatives of $f$, we differentiate the defining equation for $f$ to find
    \[
      1 = (f(x)^{1/\beta} + \beta (f(x)^{1/\beta} - 1) + \beta(f(x)^{1/\beta} - 1) f(x)) f'(x) \exp f(x).
    \]
We can use the defining equation again to eliminate the exponential term, which gives us
\[
  f'(x) = \frac1x (1 - \frac{\beta^{-1}(1 - f(x)^{-1/\beta})^{-1} + 1}{\beta^{-1}(1 - f(x)^{-1/\beta})^{-1} + 1 + f(x)}).
\]
This gives us the lemma for $k = 1$. For the higher derivatives, we differentiate by parts and apply our earlier asymptotics.
\end{proof}
Note that this lemma also shows that $r_x$ satisfies the technical condition. We apply this formula to $g(r_n)$ to get
\[
g(r_n) = \frac{n}{(\log n)^{1 + 1/\beta}} (1 + (\log n)^{-1/\beta} + (1 + \beta^{-1}) \frac{\log \log x}{\log x}(1 + o(1)))\]
and to $n \eta_1(n)$ to get
\[
  n \eta_1(n) = - \frac{n}{(\log n)^{1 + 1/\beta}} (1 + (\log n)^{-1/\beta} - (1 + \beta^{-1}) \frac{1}{\log x}(1 + o(1))). \qedhere
\]
\end{proof}

Other $g(t)$ can be computed in similar ways. Note that in the proof of Corollary~\ref{cor:exponential}, it was crucial to develop $g(r_n)$ and $n \eta_1(n)$ beyond the first term; this reflects the reduced variance of the number of cycles when there most of the cycles are of logarithmic length.

\section{Large deviation estimates} \label{sec:largedeviation}

The method developed in the previous two sections actually gives more information than a central limit theorem. In fact, it was enough for us to show that
\[
  \mathbb{E} \frac{K_{0n} - \mu_n}{\sigma_n} = \exp(\frac{s^2}{2}(1 + o(1)))
\]
for $s$ arbitrarily close to $0$, but our method applied to all $s$ in a fixed compact set. In this section we will briefly indicate how to use this extra information to prove large deviation estimates for $K_{0n}$.

Let $M(s)$ denote the moment generation function for the renormalized cycle count; i.e.~
\[
  M(s) = \mathbb{E} \exp(s \frac{K_{0n} - \mu_n}{\sigma_n})
\]
and let $\Lambda(s) = \log M(s)$ denote its logarithm. We restate the corollary of Proposition~\ref{prop:expand} as follows.
\begin{proposition}
There are functions $\sigma_n^2$, $\xi(n)$ such that the for all $s = O(\sigma_n)$, we have the estimate
\[
  \Lambda(s) = s^2 / 2 + O(\xi(n) \sigma_n^{-3}) s^3.
\]
\end{proposition}
As an immediate consequence, we also have
\[
  \Lambda'(s) = s + O(\xi(n) \sigma_n^{-3}) s^2
\]
and
\[
  \Lambda''(s) = 1 + O(\xi(n) \sigma_n^{-3}) s.
\]
Furthermore, $\Lambda'(s)$ is monotone increasing (hence injective) for such $s$.

\begin{theorem}
For all $a = O(\sigma_n)$ we have
\[
  \mathbb{P}(\abs{\frac{K_{0n} - \mu_n}{\sigma_n} - a} < \epsilon) = (1 - \epsilon^{-2} (1 + \delta)) \exp(- a^2 / 2 + O(\delta + \epsilon a))
\]
where
\[
  \delta = O(\xi(n) \sigma_n^{-3} a)
\]
and the errors are absolute.
\end{theorem}

\begin{proof}
Let $X_n := (K_{0n} - \mu_n) / \sigma_n$ and let $\eta$ denote the pdf for $X_n$. We define a pdf $\nu_s$ depending on $s \in \R$ by
\[
  d\nu_s(x) = \frac{1}{M(s)} e^{sx} d\eta(x).
\]
Then if $Y$ is a random variable with pdf $\nu_s$, we see that
\[
  \mathbb{P}(\abs{X_n - a} < \epsilon) = M(s) \mathbb{E} e^{-s Y} 1_{\abs{Y - a} < \epsilon}
\]
We want to choose $s$ so that $Y$ has mean $a$. In fact,
\[
  \mathbb{E} Y = M(s)^{-1} \mathbb{E} X_n \exp(s X_n)] = \Lambda'(s)
\]
Therefore, because $\Lambda'$ is injective we solve $s = a + O(\xi(n) \sigma_n^{-3} a^2)$. On the event that $\abs{Y - a} < \epsilon$, we see that $e^{-s Y} = e^{-sa + O(s\epsilon)}$ so that
\[
  \mathbb{P}(\abs{X_n - a} < \epsilon) = \exp(- a^2 / 2 + O(\alpha a^3 + \epsilon a)) \mathbb{P}(\abs{Y - a} < \epsilon).
\]
It is not hard to show that for $s$ chosen so that $a = \mathbb{E}Y$,
\[
  \mathbb{E} \abs{Y - a}^2 = \Lambda''(s)
\]
so that by the second moment method,
\[
  \mathbb{P}(\abs{Y - a} < \epsilon) = 1 - \epsilon^{-2} (1 + O(\xi(n) \sigma_n^{-3} a)),
\]
and the result follows.
\end{proof}

\bibliographystyle{acm}
\bibliography{literatur}

\end{document}